\documentclass[twoside,11pt,reqno]{amsart}
\usepackage{amsmath,amssymb,amscd,mathrsfs,epic,wasysym,latexsym,tikz,mathrsfs,cite,hyperref}
\usepackage{stmaryrd}
\usepackage{pb-diagram}
\usepackage[matrix,arrow]{xy}
\usepackage{commath}
\usepackage{tikz-cd}

\makeatletter

\numberwithin{equation}{section}

\newtheorem{Proposition}[equation]{Proposition}
\newtheorem{Lemma}[equation]{Lemma}
\newtheorem{Theorem}[equation]{Theorem}
\newtheorem{Corollary}[equation]{Corollary}
\theoremstyle{definition}  

\newtheorem{Remark}[equation]{Remark}

\newtheorem{Example}[equation]{Example}

\let\<\langle
\let\>\rangle

\newcommand\Comment[2][\relax]{\space\par\medskip\noindent%
   \fbox{\begin{minipage}{\textwidth}\textbf{Comment\ifx\relax#1\else---#1\fi}\newline%
        #2\end{minipage}}\medskip
}

\def\bs{\text{\boldmath$s$}}

\def\br{\text{\boldmath$r$}}
\def\b1{\text{\boldmath$1$}}

\def\bb{\text{\boldmath$b$}}

\def\m{\mathfrak{m}}

\newcommand{\Hom}{\operatorname{Hom}}

\newcommand{\End}{\operatorname{End}}

\newcommand{\Z}{\mathbb{Z}}
\newcommand{\K}{\mathbb{K}}
\newcommand{\F}{\mathbb{F}}

\newcommand{\0}{{\bar 0}}
\renewcommand{\1}{{\bar 1}}

\def\phi{{\varphi}}

\newcommand{\zc}{{\textsf{c}}}

\newcommand{\ze}{{\textsf{e}}}

\newcommand{\za}{{\textsf{a}}}

\newcommand{\zu}{{\textsf{u}}}

\newcommand{\la}{\lambda}
\newcommand{\La}{\Lambda}

\newcommand{\om}{\omega}

\newcommand{\Ind}{{\mathrm {Ind}}}

\def\rank{\mathop{\mathrm{ rank}}\nolimits}

\def\K{\mathbb K}

\newcommand{\Zig}{{\textsf{A}}}
\newcommand{\ZigB}{\tilde{\textsf{A}}}

\def\col{{\operatorname{col}}}

\def\b{\mathfrak{b}}
\def\k{\Bbbk}

\def\iso{\stackrel{\sim}{\longrightarrow}}

\def\Seq{\operatorname{Tri}}

\def\t{{\sf t}}

{\catcode`\|=\active
  \gdef\set#1{\mathinner{\lbrace\,{\mathcode`\|"8000%
  \let|\midvert #1}\,\rbrace}}
}
\def\midvert{\egroup\mid\bgroup}

\colorlet{darkgreen}{green!50!black}
\tikzset{dots/.style={very thick,loosely dotted},
         greendot/.style={fill,circle,color=darkgreen,inner sep=1.5pt,outer sep=0},
         blackdot/.style={fill,circle,color=black,inner sep=1.1pt,outer sep=0},
         graydot/.style={fill,circle,color=gray,inner sep=1.1pt,outer sep=0},
         reddot/.style={fill,circle,color=red,inner sep=1.1pt,outer sep=0},
         bluedot/.style={fill,circle,color=blue,inner sep=1.1pt,outer sep=0}
}
\def\greendot(#1,#2){\node[greendot] at(#1,#2){}}
\def\blackdot(#1,#2){\node[blackdot] at(#1,#2){}}
\def\graydot(#1,#2){\node[graydot] at(#1,#2){}}
\def\reddot(#1,#2){\node[reddot] at(#1,#2){}}
\def\bluedot(#1,#2){\node[bluedot] at(#1,#2){}}

\newenvironment{braid}{
  \begin{tikzpicture}[baseline=6mm,black,line width=.7pt, scale=0.32,
                      draw/.append style={rounded corners},
                      every node/.append style={font=\fontsize{5}{5}\selectfont}]%
  }{\end{tikzpicture}
}

\def\Grid(#1,#2){
  \draw[very thin,gray,step=2mm] (0,0)grid(#1,#2);
  \draw[very thin,darkgreen,step=10mm] (0,0)grid(#1,#2);
}

\newcommand\Tableau[2][\relax]{
  \begin{tikzpicture}[scale=0.5,draw/.append style={thick,black}]
    \ifx\relax#1\relax%
    \else 
      \foreach\box in {#1} { \filldraw[blue!30]\box+(-.5,-.5)rectangle++(.5,.5); }
    \fi
    \newcount\row\newcount\col
    \row=0
    \foreach \Row in {#2} {
       \col=1
       \foreach\k in \Row {
          \draw(\the\col,\the\row)+(-.5,-.5)rectangle++(.5,.5);
          \draw(\the\col,\the\row)node{\k};
          \global\advance\col by 1
       }
       \global\advance\row by -1
    }
  \end{tikzpicture}
}

\newcommand\YoungDiagram[2][\relax]{
  \begin{tikzpicture}[scale=0.5,draw/.append style={thick,black}]
    \ifx\relax#1\relax%
    \else 
    \foreach\box in {#1} {
      \filldraw[blue!30]\box rectangle ++(1,1);
    }
    \fi
    \newcount\row
    \row=0
    \foreach \col in {#2} {
       \draw(1,\the\row)grid ++(\col,1);
       \global\advance\row by -1
    }
  \end{tikzpicture}
}

\newenvironment{Young}{\begingroup
       \def\vr{\vrule height0.89\hoogte width\dikte depth 0.2\hoogte}
       \def\fbox##1{\vbox{\offinterlineskip
                    \hrule height\dikte
                    \hbox to \breedte{\vr\hfill##1\hfill\vr}
                    \hrule height\dikte}}
       \vbox\bgroup \offinterlineskip \tabskip=-\dikte \lineskip=-\dikte
            \halign\bgroup &\fbox{##\unskip}\unskip  \crcr }
       {\egroup\egroup\endgroup}
\def\diagram#1{\relax\ifmmode\vcenter{\,\begin{Young}#1\end{Young}\,}\else%
              $\vcenter{\,\begin{Young}#1\end{Young}\,}$\fi}

\begin{document}

\title[On maximally symmetric subalgebras]{{\bf On maximally symmetric subalgebras}}

\author{\sc Alexander Kleshchev}
\address{Department of Mathematics\\ University of Oregon\\ Eugene\\ OR 97403, USA}
\email{klesh@uoregon.edu}


\thanks{The author was supported by the NSF grant DMS-2346684. }

\begin{abstract}
Let $\k$ be a characteristic zero PID, $S$ be a $\k$-algebra and $T\subseteq S$ be a full rank subalgebra. Suppose the algebra $T$ is symmetric. It is important to know when  $T$ is a {\em maximal symmetric subalgebra}  of $S$, i.e. no $\k$-subalgebra $C$ satisfying $T\subsetneq C\subseteq S$ is symmetric. In this note we establish a useful sufficient condition for this using a notion of a quasi-unit of an algebra. This condition is used to obtain an old and a new results on maximal symmetricity for generalized Schur algebras corresponding to certain Brauer tree algebras. 
The old result was used in our work with Evseev on RoCK blocks of symmetric groups. 
The new  result will be used in our forthcoming work on RoCK blocks of double covers of symmetric groups. 
\end{abstract}

\maketitle

\section{Introduction}

Let $\k$ be a commutative (unital) ring. 
All $\k$-algebras in this note are assumed to be unital and free of finite rank as $\k$-modules. 
A $\k$-algebra $A$ is called {\em symmetric} if it possesses a {\em symmetrizing form}, i.e. a linear form $\t$ 
such that the bilinear form 
$$
(a,b)_\t:=\t(ab) \qquad(a,b\in A)
$$ 
is a symmetric perfect pairing, which means $(a,b)_\t=(b,a)_\t$ for all $a,b$ and $A\to \Hom_\k(A,\k),\ a\mapsto (a,\cdot)_\t$ is an isomorphism of $\k$-modules. If a field $\F$ is a $\k$-module we extend scalars to get the $\F$-algebra $A_\F:=\F\otimes_\k A$. 
If the $\k$-algebra $A$ is symmetric then the $\F$-algebra $A_\F$ is also symmetric.

From now on, unless otherwise stated, let $\k$ be a PID with fraction filed $\K$ of characteristic zero. 
Let $S$ be a $\k$-algebra and suppose we have a full rank (unital) $\k$-subalgebra $T\subseteq S$. The full rank assumption means that we can identify $T_\K=S_\K$, so $S$ and $T$ are two $\k$-forms of the same finite dimensional $\K$-algebra. Suppose the $\k$-algebra $T$ is symmetric. Then the $\K$-algebra $T_\K$ is also symmetric. However, the $\k$-algebra $S$ might not be symmetric. 
In fact, it is important to know when  $T$ is a {\em maximal symmetric subalgebra} of $S$, i.e. no $\k$-subalgebra $C$ satisfying $T\subsetneq C\subseteq S$ is symmetric. 

Non-trivial examples of such maximally symmetric subalgebras $T\subseteq S$ are given by $D^A(n,d)\subseteq {}'D^A(n,d)$, where $D^A(n,d)$ is a Turner's double algebra \cite{Turner} and  ${}'D^A(n,d)$ is its divided power version, see \cite[Theorem 6.6]{EK1}. This maximal symmetricity was crucially used in \cite{EK2} to prove Turner's conjecture that RoCK blocks of symmetric groups of $p$-weight~$d$ are Morita equivalent to $D^{\Zig_{p-1}}(d,d)$ for the Brauer tree algebra $\Zig_{p-1}$, cf. Example~\ref{Ex1}. 

It might be difficult to prove that $T\subseteq S$ is a maximally symmetric subalgebra. In this note we develop the techniques of \cite[\S6.2]{EK1} to get a sufficient condition which allows us to get new important examples of such subalgebras. The  new examples will be used in \cite{KlSpinEKTwo} 
to prove that RoCK blocks of double covers of symmetric groups of $\bar p$-weight $d$ are Morita equivalent to generalized Schur algebras  corresponding to the Brauer tree algebras $\tilde\Zig_{(p-1)/2}$, cf. Example~\ref{Ex2}. 

Our sufficient condition for maximal symmetricity is based on the notion of a quasi-unit. 
Let $A$ be a unital ring. An element $\xi\in A$ will be called a {\em quasi-unit} if $\xi\in Az$ 
for a central $z\in A$ implies that  
$z$ is a unit.

For the following theorem, assume that the $\k$-algebra $S$ is non-negatively graded, i.e. $S=S^0\oplus S^1\oplus\dots\oplus S^N$ and $S^iS^j\subseteq S^{i+j}$; in particular, $S^0$ is a subalgebra. 
Let $T=T^0\oplus T^1\oplus\dots\oplus T^N$ be a {\em graded subalgebra} of $S$; in particular,  $T^i=T\cap S^i$ for all~$i$. We assume that $T$ is a full rank graded subalgebra, i.e. $\rank_\k T^i=\rank_\k S^i$ for all $i$. 
We assume that $T$ possesses a {\em degree $-N$ symmetrizing form}, which means a symmetrizing form $\t$ such that $\t|_{T^i}=0$ unless $i=N$. For the intermediate subalgebra $C$ we do not assume that it is a graded subalgebra, so a priori we might have that $\sum_i (C\cap S^i)\neq C$. As for the symmetricity of $C$, we only assume that $C_{\k/\m}$ are  symmetric for all maximal ideals $\m$ of $\k$ (symmetrizing forms on the algebras $C_{\k/\m}$ a priori have nothing to do with the given symmetrizing form $\t$ on $T$). 

\vspace{2mm}
\noindent
{\bf Main Theorem.}
{\em
Let 
$\k$ be a PID with fraction field of characteristic zero, 
$S=S^0\oplus S^1\oplus\dots\oplus S^N$ be a graded $\k$-algebra and $T\subseteq S$ be a full rank graded $\k$-subalgebra with $T^0=S^0$. Suppose $T$ possesses a degree $-N$ symmetrizing form. Let\, $C$ be a $\k$-subalgebra of $S$ containing  $T$ and such that for every maximal ideal $\m$ of\, $\k$ the $(\k/\m)$-algebra $C_{\k/\m}$ is  symmetric. 
Suppose that  and there is $\xi\in S^0$ such that:
\begin{enumerate}
\item[{\rm (a)}] each $y\in S^N$ can be written in the form $y=y_1+y_2$ with $\xi y_1= 0$ and $y_2\in T^N$;
\item[{\rm (b)}] $1_{\k/\m}\otimes\xi$ is a quasi-unit in $S^0_{\k/\m}$ for every maximal ideal\, $\m$ of $\k$. 
\end{enumerate}
Then $C=T$. 
}
\vspace{2mm}

\section{Proof of the Main Theorem} 

\subsection{Reduction to the case $\k$ is a DVR}

We first note that it suffices to prove the Main Theorem in the case where $\k$ is a discrete valuation ring (DVR). 
Indeed, we may assume that $\k$ is not a field, since in that case $T=S$ and the theorem is trivial. 
Now, for every maximal ideal $\m$ of $\k$ the localization $\k_\m$ is a DVR, and so by the DVR case of the Main Theorem, $\k_\m\otimes_\k C=\k_\m\otimes_\k T$. 
Since this is true for all $\m$ we deduce that $C=T$. 

So from now on we assume that $\k$ is a DVR with fraction field $\K$ of characteristic $0$ and the maximal ideal $(\pi)$; we denote  $\F:=\k/(\pi)$. 

\subsection{Some general remarks and notation}
Let $A$ be a $\k$-algebra (as usual unital and free of finite rank as a $\k$-module). We identify  
$A_\F=A/\pi A$, with $1_\F\otimes a$ corresponding to $a+\pi A$ for all $a\in A$.  Given an $A$-(bi)module $V$, we can extend scalars to get an $A_\K$-(bi)module $V_\K=\K\otimes _\k V$ and 
an $A_\F$-(bi)module $V_\F=\F\otimes _\k V=V/\pi V$.

\begin{Lemma} \label{L2} 
Let $I$ be a two-sided ideal of $A$ such that the $A/\pi A$-bimodules $A/\pi A$ and $I/\pi I$ are isomorphic. If there is $\xi\in I$ such that $\xi+\pi A$ is a quasi-unit in the ring $A/\pi A$ then $I=A$. 
\end{Lemma}
\begin{proof}
Let  $\phi: A/\pi A\iso I/\pi I$
be an isomorphism of $A/\pi A$-bimodules. Let $z\in I$ be such that  $\phi(1_A+\pi A)=z+\pi I$. Then $z+\pi A$ is central in $A/\pi A$. Moreover, $z+\pi I$ generates $I/\pi I$ as a left $A/\pi A$-module, so 
$\xi+\pi I\in (A/\pi A)(z+\pi I)$, hence $\xi+\pi A\in (A/\pi A)(z+\pi A)$. By assumption, $\xi+\pi A$ is a quasi-unit in $A/\pi A$, so $z+\pi A$ is a unit in $A/\pi A$. It follows that $I+\pi A=A$. Hence $I=A$ by Nakayama's Lemma. 
\end{proof}

If $A=\bigoplus_{i\in \Z}A^i$ and $j\in \Z$, we denote 
$$A^{\geq j}:=\bigoplus_{i\geq j}A^i\quad\text{and}\quad A^{> j}:=\bigoplus_{i> j}A^i.
$$

\subsection{Proof of the Main Theorem for the case $\k$ is a DVR}
For $i\in\Z_{\geq 0}$, we define 
$$
C^{(i)}:= C\cap S^i\quad\text{and}\quad C^{(>i)}:=C\cap S^{>i}.
$$ 
(Since a priori $C$ is a not necessarily graded subalgebra of $S$, we cannot claim that $C=\bigoplus_i C^{(i)}$, so we avoid the notation $C^i$, $C^{>i}$, etc.)
Note that $C^{(i)}$ and $C^{(>i)}$ are pure $\k$-submodules of $C$. So we can identify $C^{(i)}_\F:=\F\otimes_\k C^{(i)}$ and 
$C^{(> i)}_\F:=\F\otimes_\k C^{(> i)}$ with $\F$-supspaces of $C_\F$. 

For every $i$, we have $T^i\subseteq C^{(i)}\subseteq S^i$, and so 
\begin{equation}\label{ERanks}
\rank_\k C^{(i)}=\rank_\k T^i=\rank_\k S^i.
\end{equation}
Taking into account that $S^0=T^0$, we also have
\begin{equation}\label{ESTC0}
T^0=C^{(0)}=S^0.
\end{equation}
This, in turn, implies
\begin{equation}\label{EDirect Sum}
C=C^{(0)}\oplus C^{(>0)}.
\end{equation}
Indeed, it is clear that the sum $C^{(0)}+ C^{(>0)}$ is direct. Moreover, given $c\in C$, we can write $c=s_0+s_{>0}$ with $s_0\in S^0=C^{(0)}$ and $s_{>0}\in  S^{>0}$. It now follows that  $s_{>0}\in  C^{(>0)}$. 

By assumption, there exists a degree $-N$ symmetrizing form $\t$ on $T$. The corresponding bilinear form $(\cdot,\cdot)_\t$ restricts to a perfect pairing between $T^j$ and $T^{N-j}$ for all $j=0,\dots,N$; in particular, $(\cdot,\cdot)_\t$ restricts to a perfect pairing between $T^0$ and $T^N$, so $\rank_\k T^0=\rank_\k T^N$, and (\ref{ERanks}), (\ref{ESTC0}) imply $\rank_\k C^{(0)}=\rank_\k C^{(N)}$, hence
\begin{equation}\label{EDimC0CN}
\dim C^{(0)}_\F=\dim C^{(N)}_\F.
\end{equation}
It is clear that $C^{(N)}$ is an ideal in $C$, so naturally a $C^{(0)}$-bimodule. Extending scalars,  $C^{(N)}_\F$ is a $C^{(0)}_\F$-bimodule. 

\vspace{3mm}
\noindent
{\sf Claim 1:}  The $C^{(0)}_\F$-bimodule $C^{(N)}_\F$ is isomorphic to the dual bimodule $(C^{(0)}_\F)^*$. 

\vspace{1mm}
\noindent Indeed, $C_\F$ is symmetric by assumption, so 
it has a symmetrizing form $\t^C\in C_\F^*$. So the bilinear form on $C_\F$ defined as 
$$(c,c')_{C}:=\t^C(cc')\qquad(c,c'\in C_\F)$$ 
is symmetric and non-degenerate. We consider the orthogonal complement $(C^{(N)}_\F)^\perp$ of $C^{(N)}_\F$ in $C_\F$ with respect to $(\cdot,\cdot)_{C}$. Note that 
$C^{(>0)}_\F C^{(N)}_\F=0$, so 
$C^{(>0)}_\F\subseteq (C^{(N)}_\F)^\perp$. 
Moreover, by (\ref{EDirect Sum}), we have $C_\F=C^{(0)}_\F\oplus C^{(>0)}_\F$. 
By dimensions, taking into account (\ref{EDimC0CN}), we now deduce that $C^{(>0)}_\F= (C^{(N)}_\F)^\perp$ and that $(\cdot,\cdot)_{C}$ restricts to a perfect pairing between $C^{(0)}_\F$ and $C^{(N)}_\F$, which easily implies the claim.

\vspace{3mm}
Extending scalars from $\k$ to $\K$, we can identify $T_\K=C_\K=S_\K$ and consider $S, T,C$, etc. as sublattices in $S_\K$. The bilinear form $(\cdot,\cdot)_{\t}$ extends to the  bilinear form $(\cdot,\cdot)_{\t,\K}$ on $T_\K=S_\K$. 
Since $(\cdot,\cdot)_{\t}$ restricts to a perfect pairing between $T^0$ and $T^N$, taking into account (\ref{ESTC0}), we get 
\begin{equation}\label{ETDual}
S^0=C^{(0)}=T^0=\{x\in S_\K^0\mid (x,y)_{\t,\K}\in\k\ \text{for all $y\in T^N$}\}. 
\end{equation}
Then
\begin{eqnarray}\label{EPerfOpp2}
 T^N&=&\{y\in S_\K^N\mid (x,y)_{\t,\K}\in\k\ \text{for all $x\in S^0$}\}.
\end{eqnarray}
We set
\begin{align}
\label{ECDual}
\tilde C^{(0)}&:=\{x\in S_\K^0\mid (x,y)_{\t,\K}\in\k\ \text{for all $y\in C^{(N)}$}\},
\\
\tilde S^0&:=\{x\in S_\K^0\mid (x,y)_{\t,\K}\in\k\ \text{for all $y\in S^N$}\}.
\end{align}
Then
\begin{eqnarray}\label{EPerfOpp1}
 C^{(N)}&=&\{y\in S_\K^N\mid (x,y)_{\t,\K}\in\k\ \text{for all $x\in \tilde C^{(0)}$}\}.
\end{eqnarray}
Since $T^N\subseteq C^{(N)}\subseteq S^N$, we have a chain of $\k$-submodules 
\begin{equation}\label{E140624}
\tilde S^0\subseteq \tilde C^{(0)}\subseteq S^0.
\end{equation}

We claim that in fact $\tilde C^{(0)}$ is an $S^0$-subbimodule of $S^0$ (this is also true for $\tilde S^0$ by a similar argument but we will not need this). Indeed, let $x\in \tilde C^{(0)}$ and $s\in S^0$. By (\ref{ESTC0}), $s\in C^{(0)}$, and so for all $y\in C^{(N)}$, we have that $sy,ys\in C^{(N)}$, whence 
$$
(xs,y)_{\t,\K}=(x,sy)_{\t,\K}\in\k
$$
and
$$
(sx,y)_{\t,\K}=(y,sx)_{\t,\K}=(ys,x)_{\t,\K}=(x,ys)_{\t,\K}\in\k,
$$ 
proving that $sx,xs\in \tilde C^{(0)}$.

Since $S^0=C^{(0)}$ by (\ref{ESTC0}), we have that $C^{(N)}$ is an $S^0$-bimodule, and the pairing $(\cdot,\cdot)_{\t,\K}$, being associative, induces an isomorphism $\tilde C^{(0)}\cong (C^{(N)})^*$ 
of $S^0$-bimodules. 
Hence,
\begin{equation}\label{EBimIso}
\tilde C^{(0)}_\F\cong ((C^{(N)})^*)_\F\cong (C^{(N)}_\F)^*
\end{equation}
 as $S^0_\F$-bimodules. 

Recall that by assumption we have an element $\xi\in S^0$
such that $1_{\F}\otimes\xi$ is a quasi-unit in $S^0_{\F}$. 
By (\ref{ESTC0}), we have $\xi\in T^0$. By the assumption (a) in the Main Theorem, if $y\in S^N$ then $y=y_1+y_2$ such that $\xi y_1=0$ and $y_2\in T^N$. So $(\xi,y)_{\t,\K}=(\xi,y_2)_{\t}\in\k$. We have proved that $\xi\in\tilde S^0$. By  (\ref{E140624}), we have 
$
\xi\in \tilde C^{(0)}.
$

Now, using (\ref{ESTC0}), Claim 1 and an isomorphism from (\ref{EBimIso}), we get
$$
S_\F^0=C^{(0)}_\F\cong (C^{(N)}_\F)^*\cong \tilde C^{(0)}_\F
$$
as $S_\F^0$-bimodules. Since we have by the previous paragraph that $
\xi\in \tilde C^{(0)}$, Lemma~\ref{L2} 
with $A=S^0$ and $I=\tilde C^{(0)}$ 
yields $\tilde C^{(0)}= S^0$. By (\ref{EPerfOpp2}) and (\ref{EPerfOpp1}), we deduce that $T^N= C^{(N)}$.  

Now suppose for a contradiction that $C\neq T$. Choose $x\in C\setminus T$ that lies in $S^{\geq j}$ with $j$ maximal possible. We can write $x=s_j+s_{j+1}+\dots+s_{N}$ with $s_i\in S^i$ for all $i$. By maximality of $j$, we have that $s_j\not\in T^j$. So we can write $s_j=c t$ for $c\in \K\setminus \k$ and $t\in T^j\setminus\pi T^j$. There exists $u\in T^{N-j}$ such that 
$
\t^T(tu)+(\pi)\neq 0
$ in $\k/(\pi)=\F$. Hence $tu\not\in\pi T^N$. Since $c\not\in\k$, we have $ctu\not\in T^N$. But $T^N=C^{(N)}$ by the previous paragraph, so $ctu\not\in C^{(N)}$. 
Moreover, $s_{j+1}u=\dots=s_Nu=0$ by degrees. Hence $xu=ctu\not\in C^{(N)}$. This is a contradiction since $x\in C$ and $u\in T^{N-j}\subseteq T\subseteq C$.

\section{Examples}
\subsection{A construction of quasi-units} 
\label{SSQU}
A key assumption in the Main Theorem is the existence of an interesting quasi-unit. In this subsection, we provide a necessary condition for an idempotent to be a quasi-unit.  

\begin{Proposition}\label{PIdQU} 
Let $A$ be a (unital) ring and let\, $1_A=e_0+e_1+\dots+e_k$ be an orthogonal idempotent decomposition. Suppose that for every $i=1,\dots,k$ we have:
\begin{enumerate}
\item[{\rm (i)}] Considering $e_iAe_0$ as an $(e_iAe_i,e_0Ae_0)$-bimodule, the natural map $e_iAe_i\to\End_{e_0Ae_0}(e_iAe_0)$ is an isomorphism. 
\item[{\rm (ii)}] There exists $a_i\in e_iAe_0$ such that $e_iAe_0=a_iAe_0$.
\end{enumerate}
Then $e_0$ is a quasi-unit in $A$.
\end{Proposition}
\begin{proof}
Suppose that $z\in A$ is central and $e_0\in Az$. We need to prove that $z$ is a unit. Note that $e_ize_j=ze_ie_j=0$ for $i\neq j$, so $z=z_0+z_1+\dots+z_k$ where $z_i=ze_i=e_iz$. 

Since  $e_0\in Az$, there exists $y_0\in A$ such that $e_0=y_0z$. Replacing $y_0$ with $e_0y_0e_0$ if necessary, we may assume that $y_0\in e_0Ae_0$. Now, $y_0$ is the inverse of the central element $z_0$ in the ring $e_0Ae_0$ (with unit $e_0$). In particular,  $y_0$ is central in the ring $e_0Ae_0$. Hence for every $i=1,\dots,k$, there is a right $e_0Ae_0$-module endomorphism of $e_iAe_0$ given by the right multiplication with $y_0$.  By (i), there is $y_i\in e_iAe_i$ such that this endomorphism is given by the left multiplication with $y_i$. In particular,
$
a_iy_0=y_ia_i$.
Therefore,
$$
z_iy_ia_i=z_ia_iy_0=za_iy_0=a_iy_0z=a_ie_0=a_i.
$$
Since the natural map in (i) is injective and $a_i$ generates $e_iAe_0$ as a right $e_0Ae_0$-module by (ii), the equality $z_iy_ia_i=a_i$ implies 
$z_iy_i=e_i$. Now, 
$$
z(y_0+y_1+\dots+y_k)=y_0z+z_1y_1+\dots+z_ky_k=e_0+e_1+\dots+e_k=1_A,
$$ 
so $z$ is a unit.
\end{proof}

\begin{Example} \label{Ex} 
{\rm 
Let $n\in\Z_{>0}$, $\k$ be an arbitrary commutative ring, and $M_n(\k)$ be the matrix algebra. For $d\in\Z_{\geq 0}$ the symmetric group $S_d$ acts on the algebra $M_n(\k)^{\otimes d}$, and the algebra of invariants  $S(n,d):=(M_n(\k)^{\otimes d})^{S_d}$ is the classical Schur algebra, see \cite{Green}. The algebra $S(n,d)$ comes with the orthogonal (weight) idempotent decomposition 
\begin{equation}\label{EId}
1_{S(n,d)}=\sum_{\la\in\La(n,d)}\xi_\la,
\end{equation}
where $\La(n,d)$ denotes the set of all compositions of $d$ with at most $n$ parts, see \cite[Section 3.2]{Green}. 
If $d\leq n$ then we have the special composition $\om:=(1,\dots,1,0,\dots,0)\in \La(n,d)$ (with $1$ repeated $d$ times) such that $\xi_\om S(n,d)\xi_\om\cong \k S_d$. Taking $e_0$ in Proposition~\ref{PIdQU} to be $\xi_\om$ (and $e_1,\dots,e_k$ to be the remaining $\xi_\la$), the assumption (i) the proposition comes from the Schur-Weyl duality \cite[(2.6c)]{Green}, and the assumption (ii) comes from the fact that $\xi_\la S(n,d)\xi_\om$, as a right module over 
$\k S_d\cong \xi_\om S(n,d)\xi_\om$, 
is the induced module $\Ind_{\k S_\la}^{\k S_d}\k$ from the trivial module $\k$. So 
if $d\leq n$ then $\xi_\om$ is a quasi-unit in $S(n,d)$. 
}
\end{Example}

More generally, let $A=A_\0\oplus A_\1$ be a $\k$-superalgebra and 
\begin{equation}\label{ESA}
S^A(n,d):=(M_n(A)^{\otimes d})^{S_d},
\end{equation}
with the action of $S_d$ on $M_n(A)^{\otimes d}$ taking into account the parity of elements, see \cite[(3.3),\,\S5.2]{EK1}. The algebra $S^A(n,d)$ comes with the weight idempotent decomposition identical to (\ref{EId}). 
The same argument as in Example~\ref{Ex} now recovers the following result proved in \cite[Lemma 5.18]{EK1}:

\begin{Corollary} \label{CXiOm} 
If $d\leq n$ then $\xi_\om$ is a quasi-unit in $S^A(n,d)$.
\end{Corollary}

\subsection{Some Turner's algebras} In this subsection, $\k$ is a PID with fraction filed $\K$ of characteristic zero. Let $A=A^0\oplus A^1\oplus A^2$ be a graded $\k$-algebra, with a degree $-2$ symmetrizing form $\t$, so we have a perfect symmetric pairing on $A$ given by $(a,b)_\t=\t(ab)$. Since the symmetrizing form $\t$ has degree $-2$, the bilinear form $(\cdot,\cdot)_\t$ restricts to a perfect pairing between $A^0$ and $A^2$. 
We consider $A$ as a superalgebra with $A_\0=A^0\oplus A^2$ and $A_\1=A^1$. 
Here are two key examples:

\begin{Example} \label{Ex1}
{\rm 
For all $\ell\in \Z_{>0}$, we define the graded $\k$-algebras $\Zig_\ell$. If $\ell =1$, we set $\Zig_1:= \Z[\zc_1]/(\zc_1^2)$, where $\zc_1$ is an indeterminate in degree $2$. For $\ell>1$, let $Q$ be the 
quiver 
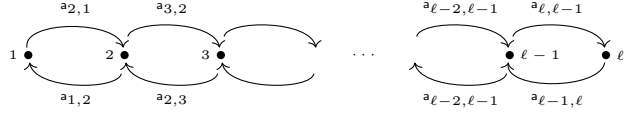
\begin{align*}
\begin{braid}\tikzset{baseline=3mm}
\coordinate (1) at (0,0);
\coordinate (2) at (4,0);
\coordinate (3) at (8,0);
\coordinate (4) at (12,0);
\coordinate (6) at (16,0);
\coordinate (L1) at (20,0);
\coordinate (L) at (24,0);
\draw [thin, black,->,shorten <= 0.1cm, shorten >= 0.1cm]   (1) to[distance=1.5cm,out=100, in=100] (2);
\draw [thin,black,->,shorten <= 0.25cm, shorten >= 0.1cm]   (2) to[distance=1.5cm,out=-100, in=-80] (1);
\draw [thin,black,->,shorten <= 0.25cm, shorten >= 0.1cm]   (2) to[distance=1.5cm,out=80, in=100] (3);
\draw [thin,black,->,shorten <= 0.25cm, shorten >= 0.1cm]   (3) to[distance=1.5cm,out=-100, in=-80] (2);
\draw [thin,black,->,shorten <= 0.25cm, shorten >= 0.1cm]   (3) to[distance=1.5cm,out=80, in=100] (4);
\draw [thin,black,->,shorten <= 0.25cm, shorten >= 0.1cm]   (4) to[distance=1.5cm,out=-100, in=-80] (3);
\draw [thin,black,->,shorten <= 0.25cm, shorten >= 0.1cm]   (6) to[distance=1.5cm,out=80, in=100] (L1);
\draw [thin,black,->,shorten <= 0.25cm, shorten >= 0.1cm]   (L1) to[distance=1.5cm,out=-100, in=-80] (6);
\draw [thin,black,->,shorten <= 0.25cm, shorten >= 0.1cm]   (L1) to[distance=1.5cm,out=80, in=100] (L);
\draw [thin,black,->,shorten <= 0.1cm, shorten >= 0.1cm]   (L) to[distance=1.5cm,out=-100, in=-100] (L1);
\blackdot(0,0);
\blackdot(4,0);
\blackdot(8,0);
\blackdot(20,0);
\blackdot(24,0);
\draw(0,0) node[left]{$1$};
\draw(4,0) node[left]{$2$};
\draw(8,0) node[left]{$3$};
\draw(14,0) node {$\cdots$};
\draw(20,0) node[right]{$\ell-1$};
\draw(24,0) node[right]{$\ell$};
\draw(2,1.2) node[above]{$\za_{2,1}$};
\draw(6,1.2) node[above]{$\za_{3,2}$};
\draw(18,1.2) node[above]{$\za_{\ell-2,\ell-1}$};
\draw(22,1.2) node[above]{$\za_{\ell,\ell-1}$};
\draw(2,-1.2) node[below]{$\za_{1,2}$};
\draw(6,-1.2) node[below]{$\za_{2,3}$};
\draw(18,-1.2) node[below]{$\za_{\ell-2,\ell-1}$};
\draw(22,-1.2) node[below]{$\za_{\ell-1,\ell}$};
\end{braid}
\end{align*}
Then $\Zig_\ell$ is the path algebra $\k Q$, generated by length $0$ paths $\ze_1,\dots,\ze_\ell$ and length $1$ paths $\za_{k,j}$, subject to the following relations:
\begin{enumerate}
\item All paths of length three or greater are zero.
\item All paths of length two that are not cycles are zero.
\item All cycles of length $2$ based at the same vertex are equal.
\end{enumerate}

The algebra $\Zig_\ell$ inherits the path length grading from $\k Q$ so that $\Zig_\ell=\Zig_\ell^0\oplus \Zig_\ell^1\oplus \Zig_\ell^2$. Setting $\zc_j:=\za_{j,j-1}\za_{j-1,j}$ or $\za_{j,j+1}\za_{j+1,j}$, we have that $\zc_1,\dots,\zc_\ell$ is a basis of  $\Zig_\ell^2$ (this statement is also true for the case $\ell=1$). Now the degree $-2$ symmetrizing form on $\Zig_\ell$ takes the value $1$ on each $\zc_1,\dots,\zc_\ell$ (and is zero on $\Zig_\ell^0\oplus \Zig_\ell^1$). 
}
\end{Example}

\begin{Example} \label{Ex2}
{\rm 
For all $\ell\in \Z_{>0}$, we define the graded $\k$-algebras $\ZigB_\ell$. Let $\tilde Q$ be the 
quiver 
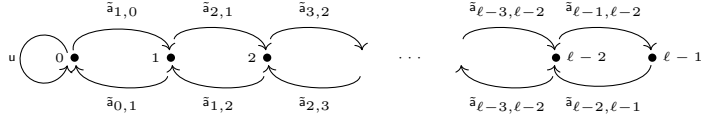
\begin{align*}
\begin{braid}\tikzset{baseline=3mm}
\coordinate (1) at (0,0);
\coordinate (2) at (4,0);
\coordinate (3) at (8,0);
\coordinate (4) at (12,0);
\coordinate (6) at (16,0);
\coordinate (L1) at (20,0);
\coordinate (L) at (24,0);
\draw [thin, black, ->] (-0.3,0.2) arc (15:345:1cm);
\draw [thin, black,->,shorten <= 0.1cm, shorten >= 0.1cm]   (1) to[distance=1.5cm,out=100, in=100] (2);
\draw [thin,black,->,shorten <= 0.25cm, shorten >= 0.1cm]   (2) to[distance=1.5cm,out=-100, in=-80] (1);
\draw [thin,black,->,shorten <= 0.25cm, shorten >= 0.1cm]   (2) to[distance=1.5cm,out=80, in=100] (3);
\draw [thin,black,->,shorten <= 0.25cm, shorten >= 0.1cm]   (3) to[distance=1.5cm,out=-100, in=-80] (2);
\draw [thin,black,->,shorten <= 0.25cm, shorten >= 0.1cm]   (3) to[distance=1.5cm,out=80, in=100] (4);
\draw [thin,black,->,shorten <= 0.25cm, shorten >= 0.1cm]   (4) to[distance=1.5cm,out=-100, in=-80] (3);
\draw [thin,black,->,shorten <= 0.25cm, shorten >= 0.1cm]   (6) to[distance=1.5cm,out=80, in=100] (L1);
\draw [thin,black,->,shorten <= 0.25cm, shorten >= 0.1cm]   (L1) to[distance=1.5cm,out=-100, in=-80] (6);
\draw [thin,black,->,shorten <= 0.25cm, shorten >= 0.1cm]   (L1) to[distance=1.5cm,out=80, in=100] (L);
\draw [thin,black,->,shorten <= 0.1cm, shorten >= 0.1cm]   (L) to[distance=1.5cm,out=-100, in=-100] (L1);
\blackdot(0,0);
\blackdot(4,0);
\blackdot(8,0);
\blackdot(20,0);
\blackdot(24,0);
\draw(0,0) node[left]{$0$};
\draw(4,0) node[left]{$1$};
\draw(8,0) node[left]{$2$};
\draw(14,0) node {$\cdots$};
\draw(20,0) node[right]{$\ell-2$};
\draw(24,0) node[right]{$\ell-1$};
 \draw(-2.6,0) node{$\zu$};
\draw(2,1.2) node[above]{$\tilde \za_{1,0}$};
\draw(6,1.2) node[above]{$\tilde \za_{2,1}$};
\draw(10,1.2) node[above]{$\tilde \za_{3,2}$};
\draw(18,1.2) node[above]{$\tilde \za_{\ell-3,\ell-2}$};
\draw(22,1.2) node[above]{$\tilde \za_{\ell-1,\ell-2}$};
\draw(2,-1.2) node[below]{$\tilde \za_{0,1}$};
\draw(6,-1.2) node[below]{$\tilde \za_{1,2}$};
\draw(10,-1.2) node[below]{$\tilde \za_{2,3}$};
\draw(18,-1.2) node[below]{$\tilde \za_{\ell-3,\ell-2}$};
\draw(22,-1.2) node[below]{$\tilde \za_{\ell-2,\ell-1}$};
\end{braid}
\end{align*}
Then $\ZigB_\ell$ is the path algebra $\k \tilde Q$, generated by length $0$ paths $\tilde \ze_0,\dots,\tilde \ze_{\ell-1}$, and length $1$ paths $\zu$ and $\tilde \za_{k,j}$, subject to the following relations:
\begin{enumerate}
\item all paths of length three or greater are zero;
\item all paths of length two that are not cycles are zero;
\item the length-two cycles based at the vertex $i\in\{1,\dots,\ell-2\}$ are equal;
\item $\zu^2=\tilde \za_{0,1}\tilde \za_{1,0}$. 
\end{enumerate}
For example, $\ZigB_1$ is the truncated polynomial algebra $\F[\zu]/(\zu^3)$. 
The algebra $\ZigB_\ell$ inherits the path length grading from $\k \tilde Q$ so that $\ZigB_\ell=\ZigB_\ell^0\oplus \ZigB_\ell^1\oplus \ZigB_\ell^2$. Setting $\tilde \zc_0=\zu^2$ and 
$\tilde \zc_j:=\tilde \za_{j,j-1}\tilde \za_{j-1,j}$ for $j=1,\dots,\ell-1$, we have that $\tilde \zc_0,\dots,\tilde \zc_{\ell-1}$ is a basis of  $\ZigB_\ell^2$. Now the degree $-2$ symmetrizing form on $\ZigB_\ell$ takes the value $1$ on each $\tilde \zc_0,\dots,\tilde \zc_{\ell-1}$ (and is zero on $\ZigB_\ell^0\oplus \ZigB_\ell^1$). 
Note that the grading on $\ZigB_\ell$ used in  \cite{KL} is different from the one used here.
}
\end{Example}

Let $n\in\Z_{>0}$ and $d\in\Z_{\geq 0}$. 
As in (\ref{ESA}), we have the algebra $S^A(n,d)$, which inherits grading from $A$ so that 
$$
S^A(n,d)=S^A(n,d)^0\oplus\dots\oplus S^A(n,d)^{2d}.
$$
Taking ${\mathfrak a}=A^0$ and ${\mathfrak c}=A^2$ in the construction of \cite[\S3B]{KM2}, we obtain the {\em Turner's algebra} $T^A_{\mathfrak a}(n,d)$ which in this paper we denote simply $T^A(n,d)$. It arises as a full rank $\k$-subalgebra of $S^A(n,d)$ with 
\begin{equation}\label{ET0S0}
T^A(n,d)^0= S^A(n,d)^0\cong S^{A^0}(n,d). 
\end{equation}
Moreover, from \cite[Lemma 6.3, Corollary 6.7]{KM2} we have

\begin{Lemma} \label{PTSymm} 
The graded $\k$-algebra $T^A(n,d)$ possesses a degree $-2d$ symmetrizing form. 
\end{Lemma}

\begin{Theorem} \label{TTASA} 
Let $d\leq n$. Then the subalgebra $T^A(n,d)\subseteq S^A(n,d)$ is maximally symmetric.
\end{Theorem}
\begin{proof}
We apply the Main Theorem with $T=T^A(n,d)$ and $S=S^A(n,d)$. The required symmetrizing form on $T$ exists by Lemma~\ref{PTSymm}, and the assumption $T^0=S^0$ comes from (\ref{ET0S0}). For the element $\xi$ we take the element $\xi_\om\in S^A(n,d)^0\cong S^{A^0}(n,d)$, which by Corollary~\ref{CXiOm} satisfies the assumption (b) of the Main Theorem. 

To check the assumption (a) of the Main Theorem, note that, in the notation of \cite[(3.2),(5.4)]{KM2}, we have 
$\xi_\om=\xi_{12\cdots d,12\cdots d}^{1_{A}1_{A}\cdots1_{A}}$. Moreover, using the notation of \cite[\S2B,\,(3.2)]{KM2},
for any $(\bb,\br,\bs)\in\Seq^B(n,d)$, we have $\xi_\om\xi_{\br,\bs}^{\bb}=0$ unless, up to permutation, we have $r_1=1,\dots,r_d=d$, in particular $r_1,\dots,r_d$ are all distinct, whence $\xi_{\br,\bs}^{\bb}=\eta_{\br,\bs}^{\bb}$, see \cite[(3.9)]{KM2}. But $\eta_{\br,\bs}^{\bb}\in T^A(n,d)$. Since by \cite[Lemma 3.3]{KM2}, there is a basis of $S^A(n,d)$ consisting of  $\xi_{\br,\bs}^{\bb}$'s , the assumption (a) follows. 
\end{proof}

\begin{Corollary} 
Let $d\leq n$. 
The subalgebras $T^{\Zig_\ell}(n,d)\subseteq S^{\Zig_\ell}(n,d)$ and $T^{\ZigB_\ell}(n,d)\subseteq S^{\ZigB_\ell}(n,d)$ are maximally symmetric.
\end{Corollary}

\begin{Remark} 
{\rm 
The Main Theorem also implies Theorem 6.6 of \cite{EK2}, by which it was motivated. That theorem claims that the Turner's double algebra $D^A(n,d)$ is a maximally symmetric subalgebra of the divided power version ${}'D^A(n,d)$ of Turner's double algebra. The proof is similar to that of  Theorem~\ref{TTASA} but uses the standard grading of ${}'D^A(n,d)$ as defined in \cite[\S4.5]{EK1} and the symmetrizing form on $D^A(n,d)$ defined in \cite[\S4.6]{EK1}. 
}
\end{Remark}


\begin{thebibliography}{ABC}

\bibitem
{EK1}
A.\ Evseev and A.\ Kleshchev, Turner doubles and generalized Schur algebras, {\em Adv. Math.} {\bf 317} (2017), 665--717.

\bibitem
{EK2}
A.\ Evseev and A.\ Kleshchev, Blocks of symmetric groups, semicuspidal KLR algebras and zigzag Schur-Weyl duality, {\em Ann. of Math. (2)} {\bf 188} (2018), 453--512.

\bibitem
{Green}
J.A.\ Green, {\em Polynomial Representations of $GL_n$}, 2nd edition, Springer-Verlag, 2007.


\bibitem
{KlSpinEKTwo}
A. Kleshchev, 
RoCK blocks of double covers of symmetric groups and generalized Schur algebras, {\em in preparation}. 

\bibitem
{KL}A.\ Kleshchev and M.\ Livesey, RoCK blocks for double covers of symmetric groups and quiver Hecke superalgebras, \textit{Mem.\ Amer.\ Math.\ Soc.}, to appear; \texttt{arXiv:2201.06870}.

\bibitem
{KM2}
A. Kleshchev and R. Muth, Generalized Schur algebras, {\em Algebra Number Theory} {\bf 14} (2020), 501--544. 

\bibitem
{Turner}
W. Turner, Rock blocks, {\em Mem. Amer. Math. Soc.} {\bf 202}, no. 947 (2009), viii+102.


\end{thebibliography}
\end{document}